\documentclass[10pt,reqno]{amsart}   % For Latex2e
\usepackage{amssymb,amscd,latexsym}   % For Latex2e
\usepackage{amsmath}
\usepackage{amsthm,times}
\usepackage[all]{xy}
\usepackage{epsfig,graphicx}
\usepackage{graphicx,color}
\usepackage[utf8]{inputenc}
\usepackage{float}
\usepackage{tikz}
\usepackage{enumitem}

\usepackage{tikz-cd}
\newcommand{\llar}{-\kern-5pt-\kern-5pt\longrightarrow}

\usepackage{cases}

\usepackage{multirow}
\usepackage{makecell} % break a line in a columm of a matrix
\newtheorem{Theorem}{Theorem} %[section]
\newtheorem{Lemma}[Theorem]{Lemma}

\newtheorem{prop}[Theorem]{Proposition}
\newtheorem{Remark}[Theorem]{Remark}
\newtheorem{Example}[Theorem]{Example}

\DeclareMathOperator{\Hom}{Hom}

\DeclareMathOperator{\Image}{Im}

\DeclareMathOperator{\ext}{Ext}
\DeclareMathOperator{\gl}{GL}

\DeclareMathOperator{\rank}{rank}

\DeclareMathOperator{\E}{\rm E}

\newcommand{\op}[1]{{\mathcal O}_{\mathbb{P}^{#1}}}

\def\i{{\rm i}\,}
\def\ii{{\rm ii}\,}
\def\iii{{\rm iii}\,}
\def\iv{{\rm iv}\,}
\def\v{{\rm v}\,}

\def\EE{{\mathcal{E}}}

\def\ker{{\rm ker}\,}

\def\restr{{\kern-1pt\restriction\kern-1pt}}

\def\X{{\mathcal X}}

\def\E{{\mathcal E}}

\keywords{}

\subjclass[2010]{}

\begin{document}

\title[Classification of minimal monads and a new component of $\mathcal{B}(9)$]{Classification of monads and a new moduli component of stable rank 2 bundles 
on $\mathbb{P}^3$ with even determinant and $c_2=9$.}

%\title{The case c_1=-1 and c_2=10}
%\date{August 2022}
\author{Aislan Leal Fontes}
\address{Departamento de Matem\'atica, UFS - Campus Itabaiana. Av. Vereador Ol\'impio Grande s/n, 49506-036 Itabaiana/SE, Brazil}
\email{aislan@ufs.br}
%
%%\author{Marcos Jardim}
%%\address{IMECC - UNICAMP \\ Departamento de Matem\'atica \\
%%Rua S\'ergio Buarque de Holanda, 651\\ 13083-970 Campinas-SP, Brazil}
%%\email{jardim@unicamp.br}

%\title{Monads with homotopy}
%\author{ailfontes }
%\date{July 2021}

%\begin{document}

\maketitle
\begin{abstract}
The goal of this paper is to classify all minimal monads whose cohomology is a stable rank 2
bundle on $\mathbb{P}^3$ with Chern classes $c_1=0$ and $c_2=9$, with possible exception 
of two non-negative minimal monads, and thus we extend the classification of the minimal monads made by 
Hartshorne and Rao in \cite[Section 5.3]{HR91} when $c_2\leq8$. We also prove the existence of a new 
component of the moduli space $\mathcal{B}(9)$ which is distinct from the Hartshorne and Ein components.
\end{abstract}

\tableofcontents
\section{Introduction}
In \cite{M77}, Maruyama constructed the moduli space $\mathcal{B}_X(c_1, c_2,\cdots c_r)$ of 
stable rank $r$ sheaves with fixed Chern classes $c_1, c_2,\cdots c_r$ on a locally Noetherian 
scheme $X$ and this moduli space has structure of algebraic 
quasi-projective variety. In this paper, let us denote $\mathcal{B}(e,m)$ the moduli space of the 
stable rank 2 vector bundles on $X=\mathbb{P}^3$ with 
$c_1=e, c_2=m$ that up to normalization, it is sufficient to consider stable rank 2 bundles with 
$c_1=-1$ or $c_1=0$. Additionally, $c_2$ is even when $c_1=-1$. With the desire to study this 
moduli space, we recall that there are three ways to produce a stable rank 2 bundle 
on $\mathbb{P}^3$. Due to Horrocks \cite{Ho64}, every vector bundle on $\mathbb{P}^3$ is
given as the cohomology of a monad whose terms are summands of line bundles. We can also study stable rank 2 bundles 
from their spectra as well as analyze the irreducible components of the moduli space of stable vector 
bundles with Chern classes fixed. We observe that there is no straightforward 
relation between these three classification schemes: a
monad can be associated with two vector bundles with different spectra; a single spectrum
can be associated with different monads; and an irreducible component of the moduli space
may contain points corresponding to different spectra and monads.

Hartshorne and Rao in \cite{HR91} showed that a sequence of $c_2$ integers, $c_2\leq20$, satisfying the conditions 
\ref{itema1}, \ref{itema2} and \ref{itema4}  is realized as spectrum of a stable rank 2 bundle 
$\mathcal{E}$ on $\mathbb{P}^3$ with $c_1(\mathcal{E})=0, c_2(\mathcal{E})=c_2$ and they also determine in \cite{HR91} all possible 
minimal monads with $c_2\leq8$, completing the work \cite{Bar77} which assumed stable bundles on 
$\mathbb{P}^3$ such that its Rao module does not have generators of degree $\geq0$. In this paper, we extend the 
considerations made by Hartshorne and Rao determining all possible minimal monads when $c_2=9$, with possible 
exception of two non-negative minimal monads cf. subsection \ref{sec4}. We use this classification of minimal 
monads and the formula 33 of \cite[p.g. 23]{MF2021} to show the existence of a new component of the 
moduli space of stable rank 2 bundles on $\mathbb{P}^3$ with 
$c_1=0, c_2=9$. 

 We can see in \cite[Subsection 5.2]{HR91} that the moduli spaces 
$\mathcal{B}(0,1)$ and $\mathcal{B}(0,2)$ are irreducible of dimension 5 and 13, respectively, while in \cite{ES} and \cite{C83} is proved that
$\mathcal{B}(0,3)$ and $\mathcal{B}(0,4)$ have two irreducible components. The moduli space
$\mathcal{B}(0,5)$ is completely characterized in \cite{AJTT} and in \cite{TV2019} a new infinity series is constructed 
whose elements are irreducible components of the moduli space $\mathcal{B}(0,n)$, distinct of
the Hartshorne and Ein components, for infinite values of $n$. For simplicity, $\mathcal{B}(m)$ stands the moduli 
space of the stable rank 2 vector bundles on $\mathbb{P}^3$ with $c_1=0, c_2=m$. 

Now we talk about the sections of this paper. In addition to addressing the three ways to classify 
stable rank 2 bundles on $\mathbb{P}^3$, we list in Section \ref{section2} all the possible spectra 
of a stable rank 2 bundle on $\mathbb{P}^3$ with $c_1=0, c_2=9$ and we collect in Lemma 2 a 
characterization of the minimal Horrocks monad whose cohomology is a stable 
rank 2 bundle with even determinant. In Section \ref{section3} we list all possible minimal monads 
whose cohomology is a stable rank 2 bundle on $\mathbb{P}^3$ with $c_1=0, c_2=9$. 

In Section \ref{section4} we use Theorem \ref{teo4} and Lemma \ref{lema48} to list all the existing minimal 
Horrocks monads whose cohomology is a stable 
rank 2 bundle  with $c_1=0, c_2=9$, with possible exception of two nonegative minimal monads, and 
thus we extending the Table of the subsection 5.3 of \cite{HR91}. In Section \ref{sec5} we explicitly show the 
existence of an infinite family of positive minimal Horrocks monad and we prove that the moduli space 
$\mathcal{B}(9)$ has a new component, this one is distinct of the Hartshorne and Ein components.  

%%%%
%%%%%
\section{Constructing vector bundles of rank $2$ on $\mathbb{P}^3$}\label{section2}
There are three forms to construct vector bundles on $\mathbb{P}^3$. Due to Horrocks \cite{Ho64}, every vector 
bundle on $\mathbb{P}^3$ is given as cohomology of a minimal monad whose terms are sums of line 
bundles. The second form to obtain a stable vector bundle $\mathcal{\E}$ is through the concept of a spectrum 
that encodes the information of the cohomology groups $H^1_*(\mathbb{P}^3,\mathcal{\E})$ and
$H^2_*(\mathbb{P}^3,\mathcal{\E})$, as observed in Section \ref{sub2} of this
paper. We also obtain a rank 2 vector bundle through the Serre correspondence.

\subsection{Minimal Horrocks monads}
A \textit{monad} on $\mathbb{P}^3$ is a complex

\begin{equation}\label{eq222}
\mathbf{M}:\ \ \  \ \mathcal{A}\stackrel{\alpha}{\longrightarrow}\mathcal{B}\stackrel{\beta}{\longrightarrow}\mathcal{C},
\end{equation}
of vector bundles on $\mathbb{P}^3$ such that the map $\alpha$ is injective and $\beta$ is surjective. The sheaf 
$\mathcal{E}:=\ker \beta/\Image\alpha$ is the \textit{cohomology of the monad} $\mathbf{M}$ of \eqref{eq222}. Let us assume that 
the morphism $\beta$ in \eqref{eq222} is locally left invertible, and so the cohomology of the 
monad $\EE$ of $\mathbf{M}$ is a vector bundle. Furthermore, due to Horrocks (see \cite{Ho64}) follows that all vector bundles on $\mathbb{P}^n$ can be obtained as cohomology of a 
monad of the form \eqref{eq222}, where the vector bundles $\mathcal{C}, \mathcal{B}$ 
and $\mathcal{A}$ are sums of line bundles. 

We assume that the morphism $\alpha$ is locally left invertible, so that $\mathcal{E}$ is a vector 
bundle. A monad $\mathbf{M}$ of the form $\eqref{eq222}$ is called \textit{minimal} if the entries of the 
associated matrices to the maps $\alpha$ and $\beta$, as homogeneous forms, are not nonzero 
constants or, in other words, if no direct summand of $\mathcal{A}$ is the 
image of a line subbundle of $\mathcal{B}$ and if no direct summand of $\mathcal{C}$ goes 
onto a direct summand of $\mathcal{B}$. 
%%%%
%%
In addition, $\mathbf{M}$ is said to be
\textit{homotopy free} if
$$ \Hom(\mathcal{C},\mathcal{B})=\Hom(\mathcal{B},\mathcal{A})=0. $$
We recall that a minimal monad as in display \eqref{eq222} is said: \textit{positive} if all summand of $\mathcal{C}$ has positive degree, \textit{non-negative} if the summands of $\mathcal{C}$ have non-negative degree and \textit{negative} if $\mathcal{C}$ has at least a summand of negative degree. By considering $\EE$ be a stable rank 2 bundle on $\mathbb{P}^3$ with even determinant there is a unique isomorphism 
$f:\EE\longrightarrow\EE^{*}$ with symplectic structure which means $f^{*}=-f$. In parallel to the description of minimal monads whose cohomology is a stable rank 2 bundle in $\mathbb{P}^3$ with $c_1=-1$ made in \cite[Lemma 3]{MF2021} we have the following characterization for minimal Horrocks monads whose cohomology is a stable bundle with even determinant.

\begin{Lemma}\label{lema3}
Let $\mathcal{E}$ be a stable rank 2 bundle on $\mathbb{P}^3$ with $c_1(\mathcal{E})=0$. There are two 
sequences of integers $\boldsymbol{a}=(a_1,\dots,a_s)$ and $\boldsymbol{b}=(b_1,\dots,b_{s+1})$ such 
that $\mathcal{E}$ is the cohomology of a monad of the form
\begin{equation}\label{eq6}
\bigoplus_{i=1}^{s}\op3(-a_i) \stackrel{\alpha}{\longrightarrow} \bigoplus_{j=1}^{s+1}
\Big(\op3(b_j)\oplus\op3(-b_j)\Big) \stackrel{\beta}{\longrightarrow} \bigoplus_{i=1}^{s}\op3(a_i),
\end{equation}
where we order the tuples $\boldsymbol{a}$ and $\boldsymbol{b}$ as $a_1\le \cdots \le a_s$ and 
$0\le b_1\le \cdots \le b_{s+1}$. 
Furthermore, 
\begin{equation}\label{c_2}
c_2=\displaystyle\sum_{i=1}^{s}a_i^2-\displaystyle\sum_{i=1}^{s+1}b_i^2.
\end{equation}
\end{Lemma}
\begin{Example}
\textup{The minimal monads of the form
$$\op3(-c)\stackrel{\alpha}{\longrightarrow}\op3(a)\oplus\op3(-a)\oplus\op3(b)\oplus\op3(-b)\stackrel{\beta}{\longrightarrow}\op3(c),$$
with $c>a\geq b\geq0$, are called \textit{Ein monads} while minimal monads of the form
$$\op3(-1)^{\oplus s}\stackrel{\alpha}{\longrightarrow}\oplus\op3^{2s+2}\stackrel{\beta}{\longrightarrow}\op3(1)^{\oplus s}, s\geq1,$$
are called \textit{Hartshorne monads}.}
\end{Example}

%%%%
Under the hypothesis of Lemma \ref{lema3}, if $\mathcal{E}$ is a stable bundle and its Rao module 
$M:=H^1_*(\mathcal{E})$ has a minimal free presentation of the form
$$ \cdots\rightarrow F_1\rightarrow F_0\rightarrow M\rightarrow0,$$
then $\rank(F_0)=s$ and from \cite[Proposition 2.2]{R84}, we have $\rank(F_1)=2s+2$.
Conversely, by repeating the arguments of Hartshorne and Rao in \cite[Proposition 3.2]{HR91}, we get 
that if $M$ admits a minimal free presentation as above, then $\mathcal{E}$ is given as cohomology 
of a minimal Horrocks monad as in \eqref{eq222} where $F_0=H^0_*(\mathcal{A})$ and 
$F_1=H^0_*(\mathcal{B})$, this is, $\mathcal{A}$ and $\mathcal{B}$ are the 
sheafifications of the free modules $F_0$ and $F_1$, respectively.

%%%%%%%%%
%%%%%%
\subsection{Spectrum possibilities of a stable bundle with even determinant and $c_2=9$}\label{sub2}
Let $\mathcal{E}$ be a rank $2$ stable vector bundle on $\mathbb{P}^3$ such that $c_1(\mathcal{E})=0$ 
and $c_2(\mathcal{E})=n$. The \textit{spectrum} of $\mathcal{E}$ is a unique 
multiset of integers $\X(\mathcal{E})=\{k_1,k_2,\cdots,k_n\}$ satisfying the 
following properties
\begin{enumerate}[label=(\alph*)]
\item \label{itemb1} $h^1(\mathbb{P}^3,\mathcal{E}(l)) = h^0(\mathbb{P}^1,\mathcal{H}(l+1))$ for $l\leq-1$ and
\item $h^2(\mathbb{P}^3,\mathcal{E}(l)) = h^1(\mathbb{P}^1,\mathcal{H}(l+1))$ for $l\geq-2$,
\end{enumerate}
where $\mathcal{H}=\bigoplus\op1(k_i)$. If a sequence of integers $\X=\{k_1,k_2,\cdots,k_{n}\}$ is the spectrum of a rank $2$ even 
stable bundle, then the following properties hold c.f. \cite[Section 7]{H80}:
\begin{enumerate}[label=\textbf{S.\arabic*}]
\item \label{itema1} (Symmetry)$\{k_i\}=\{-k_i\}$;
\item \label{itema2} (Connectedness) Any integer $k$ between two integers of $\X$ also belongs to spectrum $\X$;
\item\label{itema4} If $k=\max\{-k_i\}$ and there is an integer $u$ with $-k\leq u\leq-2$ that 
occurs just once in $\X$, then each $k_i\in\X$ with $-k\leq k_i\leq u$ 
occurs exactly once, see \cite[Proposition 5.1]{H82}.\\
\end{enumerate}
Another spectrum property $\X$ of a stable bundle $\EE$ which we will need is
\begin{equation}\label{S4}
n_l:=h^1(\EE(-l))-h^1(\EE(-l-1))=\#\{k_j\in\X:k_j\geq l-1\}, l\geq1.    
\end{equation}
%%%%It is important to comment that Hartshorne and Rao in \cite{HR91} showed that any 
%%%%sequence of $c_2$ integers satisfying the conditions \eqref{itema1}--\eqref{itema4} is realized 
%%%%as spectrum of an even stable bundle with second Chern equal to 
%%%%$c_2$ whenever $c_2\leq19$ which is not true to odd stable bundles c.f. \cite[Proposition 17]{MF2021} 
%%%%and \cite[Theorem ...]{MF2022}. 
From the properties 
of spectrum follows that the spectrum $\X^{c_2}$ of $\EE$ can be written as
\begin{equation}\label{eq3}
\mathcal{X}^{c_2}=\{{(-k)}^{s(k)}...,0^{s(0)},\cdots, k^{s(k)}\}.
\end{equation}
If we take $c_2=9$ we can list all spectrum possibilities, see Table \ref{spectra}. 
%\item \label{itema3}$n_l=h^1(\EE(l))-h^1(\EE(l-1))=\#\{k_j\in\X|k_j\geq l-1\}$, for $l\leq-1$;

%%%%%
\begin{table}[ht]%\label{table:spec}
\begin{tabular}{|p{10.0cm}|}\hline
 \begin{center}Spectrum\end{center}\\
\hline
$\X_1^{9}=\{0^9\}$\\
\hline
$\X_2^{9}=\{\cdots,0^7, 1\}$\\
\hline
$\X_3^{9}=\{\cdots,0^5, 1^2\}, \X_4^{9}=\{\cdots,0^5, 1, 2\}$\\
\hline
$\X_5^{9}=\{\cdots,0^3, 1^2, 2\}, \X_6^{9}=\{\cdots,0^3, 1, 2, 3\}$\\
\hline
$\X_7^{9}{=}\{\cdots,0, 1^4\}, \X_8^{9}{=}\{\cdots,0, 1^3, 2\}, \X_{9}^{9}{=}\{\cdots,0, 1^2, 2^2\}$,
$\X_{10}^{9}{=}\{\cdots,0, 1^2, 2, 3\}, \X_{11}^{9}{=}\{\cdots,0, 1, 2, 3, 4\}$.\\
\hline
\end{tabular}
\medskip
\caption{Possible spectra for stable rank $2$ vector bundles with $c_1=0$ and $c_2(\mathcal{E})=9$.}
\label{spectra}
\end{table}
\subsection{Serre Construction}\label{serre}
If we consider a rank 2 vector bundle $\mathcal{G}$ on $\mathbb{P}^3$ with first Chern class $c_1$ and 
$s\in H^0(\mathbb{P}^3,\mathcal{G}(t))$, for some $t\in\mathbb{Z}$, is a global section whose zero set has codimension 2, 
then $s^{\vee}$ induces an exact sequence
\begin{equation}\label{serre2}
 0\rightarrow\op3\rightarrow\mathcal{G}(t)\stackrel{s^{\vee}}{\rightarrow}\mathcal{I}_C(c_1+2t)\rightarrow0,   
\end{equation}
where $C$ is a Cohen-Macaulay curve in $\mathbb{P}^3$ equal to vanish locus of $s$. The exact sequence in 
\eqref{serre2} corresponds to an element $\eta\in\ext^1(\mathcal{I}_C(c_1+2t),\op3)\simeq H^0(\omega_C(4-2t-c_1))$ which generates the sheaf 
$\omega_C(4-2t-c_1)$. Conversely, given a pair $(C,\eta)$ where $C$ is a Cohen-Macaulay curve on 
$\mathbb{P}^3$ and $\eta$ is a non-vanishing global section of $\omega_C(4-2t-c_1)$ we get a rank 2 
bundle $\mathcal{F}$ on $\mathbb{P}^3$ and a global section $s$ of $\mathcal{F}(t)$ which zero set 
is $C$. Additionally, the Chern class of $\mathcal{F}$ are given by
$$c_1(\mathcal{F})=t \mbox{ and } c_2(\mathcal{F})=\deg C.$$
The above content is the \textit{Serre correspondence} and for more details see \cite{H78}. This Serre correspondence was 
generalized for rank 2 reflexive sheaves on $\mathbb{P}^3$ and can be found in \cite{H80}. We note that 
$\mathcal{G}$ is stable if and only if $H^0(\mathcal{I}_C(t+c_1))=0$.
%%%%%
\section{Possible positive minimal Horrocks monads}\label{section3}
Let $\EE$ be a stable rank $2$ vector bundle on $\mathbb{P}^3$ with $c_1=0$ and $c_2=9$. 
%From the properties 
%of spectrum follows that the spectrum $\X^{c_2}$ of $\EE$ can be written as
%\begin{equation}\label{eq3}
%\mathcal{X}^{c_2}=\{{(-k)}^{s(k)}...,0^{s(0)},\cdots, k^{s(k)}\}.
%\end{equation}
If we denote $A=\mathbf{k}[X_0,X_1,X_2, X_3]$ to be the ring of homogeneous polynomials, then 
$M:=H_*^1(\EE)=\displaystyle\bigoplus_{l\in\mathbb{Z}}H^1(\EE(l))$ is called \textit{first cohomology module} 
of $\EE$ which is also a module over $A$. For each integer $l$ let's consider $m_l=h^1(\EE(l))=\dim M_l$ and
$$
\rho(l)=\dim\left[H^1(\EE(l))/\op3(1)\otimes H^1(\EE(l-1))\right],
$$
that is, $\rho(l)$ is the number of minimal generators for $M$ in degree $l$. Now we apply the following 
result, see \cite[Proposition 3.1]{HR91}, to list the possibilities 
of positive minimal Horrocks monads 
associated to bundles of rank two on $\mathbb{P}^3$ with $c_1=0$ and $c_2=9$ and spectrum fixed.
\begin{Theorem}\label{possibles}
Let $\mathcal{E}$ be a stable vector bundle of rank $2$ on $\mathbb{P}^3$ with $c_1(\mathcal{E})=0$ 
and spectrum as in display \eqref{eq3}. %Besides the equality in display \eqref{eq3'},
We have, 
\begin{equation}
\rho(-k-1)=m_{-k-1}=s(k)
\end{equation}
and for $0\leq i < k$,
\begin{equation}\label{ineq1}
s(i)-2\displaystyle\sum_{j\geq i+1}s(j)\leq\rho(-i-1)\leq s(i)-1.
\end{equation}
\end{Theorem}
Fixed a spectrum $\X$, we wish list all minimal 
Horrocks monads whose cohomology is a rank two stable bundle $\EE$ with $c_1=0, c_2=9$ and 
spectrum $\X$. If we consider the spectrum $\X_1^9$ then $k=0$ and the only possibility of minimal Horrocks monad is 
$$9\cdot\op3(-1)\stackrel{\alpha}{\longrightarrow}
20\cdot\op3\stackrel{\beta}{\longrightarrow}
9\cdot\op3(1),
$$
whose cohomology is an Instanton bundle. For the spectrum $\X_{11}^9$, a direct 
verification provides the minimal monad (unique)
$$\op3(-5)\stackrel{\alpha}{\longrightarrow}
\op3(4)\oplus2\cdot\op3\oplus\op3(-4)\stackrel{\beta}{\longrightarrow}
\op3(5),
$$
which is an Ein type monad. Its existence is proved in \cite{Ein88}. To the other spectra the number 
of possible minimal Horrocks monadsgrows as the number of minimal generators 
of $H_*^1(\EE)$ varies, so we first list 
the number of minimal generators of $H_*^1(\EE)$ in 
the Table \ref{terms} according to Theorem \ref{possibles}. 
%We recall that the second Chern class 
%of a rank 2 stable bundle on $\mathbb{P}^3$ with $c_1=0$ is given by 
%%%%%
%%%%
\begin{table}[ht]
    \centering
    \begin{tabular}{|c|c|c|c|c|}
    \hline
     Spectrum & $k$& $\rho(-k-1)$& $i$& $\rho(-i-1)\in$\\
     \hline
     $\X_2^{9}$& 1& 1& 0& $\{5, 6\}$\\
     \hline
   $\X_3^{9}$& 1& 2& 0& $\{1, 2, 3, 4\}$\\
   \hline
   \multirow{2}{*}{$\X_4^{9}$}&\multirow{2}{*}{2}&\multirow{2}{*}{1}& 0 &$\{1, 2, 3, 4\}$\\
   \cline{4-5}
   &&& 1& $\{0\}$\\
   \hline
   \multirow{2}{*}{$\X_5^{9}$}&\multirow{2}{*}{2}&\multirow{2}{*}{1}& 0 &$\{0, 1, 2\}$\\
   \cline{4-5}
   &&& 1& $\{0, 1\}$\\
   \hline
   \multirow{3}{*}{$\X_{6}^{9}$}&\multirow{3}{*}{3}&\multirow{3}{*}{1}& 0&$\{0,1, 2\}$\\
   \cline{4-5}
&&&1&$\{0\}$\\
\cline{4-5}
&&&2&$\{0\}$\\
   \hline
$\X_7^{9}$ &1 &4 & 0&$\{0\}$\\
   \hline
   \multirow{3}{*}{$\X_{8}^{9}$}&\multirow{3}{*}{2}&\multirow{3}{*}{1}& 0&$\{0\}$\\
   \cline{4-5}
&&&1&$\{1, 2\}$\\
   \hline
   \multirow{2}{*}{$\X_9^{9}$}&\multirow{2}{*}{2}&\multirow{2}{*}{2}& 0 &$\{0\}$\\
   \cline{4-5}
   &&& 1& $\{0, 1\}$\\
   \hline
      \multirow{3}{*}{$\X_{10}^{9}$}&\multirow{3}{*}{3}&\multirow{3}{*}{1}& 0&$\{0\}$\\
   \cline{4-5}
&&&1&$\{0, 1\}$\\
\cline{4-5}
&&&2&$\{0\}$\\
   \hline
    \end{tabular}
    \caption{Number of minimal generators of $H_*^1(\EE)$}
    \label{terms}
\end{table}
%%%%
%%%%%%%%

In the following, we organize the possibilities of positive minimal Horrocks monads. To simplify the notation we consider in the Equation \eqref{c_2}
$$B_s:=\displaystyle\sum_{j=1}^{s+1}b_j^2.$$

%Ourr goal is to list all possible positive minimal Horrocks monads with a fixed spectrum. For this, we 
%take each possibility of number of minimal generators of $H_*^1(\EE)$ listed on the Table 
%\ref{terms} (that is for each sequence of integers $\mathbf{a}$) and we substitute on the 
%equation \eqref{eq1} to find its solutions which are the sequences of integers 
%$\mathbf{b}$. With the desire to simplify the notation, in the equation \eqref{eq1} we denote
%$$B_s:=\displaystyle\sum_{j=1}^{s+1}b_j(b_j+1).$$ 
%Finally below are all the possibilities of positive minimal Horrocks monads for each fixed spectrum c.f. Table \ref{terms}.

\vspace{1cm}
$\left(\mathbf{\X_2^{9}}, \rho(-2)=1\right)$
\begin{itemize}
    \item$\rho(-1)=5$ imply $\boldsymbol{a}=2, 1^5$ and
    $$B_6=0\Leftrightarrow \boldsymbol{b}=0^{14}.$$
    \item $\rho(-1)=6$ imply $\boldsymbol{a}=2, 1^6$ and
    $$B_7=1\Leftrightarrow \boldsymbol{b}=1, 0^{14}.$$
\end{itemize}
$\left(\mathbf{\X_3^{9}}, \rho(-2)=2\right)$
\begin{itemize}
    \item$\rho(-1)=1$ imply $\boldsymbol{a}=2^2, 1$ and
    $$B_3=0\Leftrightarrow \boldsymbol{b}=0^{8}.$$
    \item $\rho(-1)=2$ imply $\boldsymbol{a}=2^2, 1^2$ and
    $$B_4=1\Leftrightarrow \boldsymbol{b}=1, 0^{8}.$$
    \item $\rho(-1)=3$ imply $\boldsymbol{a}=2^2, 1^3$ and
    $$B_5=2\Leftrightarrow \boldsymbol{b}=1^2, 0^{8}.$$
    \item $\rho(-1)=4$ imply $\boldsymbol{a}=2^2, 1^4$ and
    $$B_4=3\Leftrightarrow \boldsymbol{b}=1^3, 0^{8}.$$
\end{itemize}
%%%%%%%%%
%%
%%
$\left(\mathbf{\X_4^{9}}, \rho(-3)=1\right)$

\begin{itemize}
    \item$\rho(-1)=1$ imply $\boldsymbol{a}=1, 3$ and
    $$B_2=1\Leftrightarrow \boldsymbol{b}=1,0^4.$$
    \item $\rho(-1)=2$ imply $\boldsymbol{a}=3,1^2$ and
    $$B_3=2\Leftrightarrow \boldsymbol{b}=1^2, 0^4.$$
    \item $\rho(-1)=3$ imply $\boldsymbol{a}=3,1^3$ and
    $$B_4=3\Leftrightarrow \textcolor{blue}{\boldsymbol{b}=1^3,0^4}.$$
    \item $\rho(-1)=4$ imply $\boldsymbol{a}=3,1^4$ and
    $$B_5=4\Leftrightarrow \textcolor{blue}{\boldsymbol{b}=1^4,0^4} \mbox{ or } \boldsymbol{b}=2, 0^{10}.$$
\end{itemize}
%%%%%
%%%%%%
$\left(\mathbf{\X_5^{9}}, \rho(-3)=1\right)$
\begin{itemize}
    \item $\rho(-1)=\rho(-2)=0$ imply $\boldsymbol{a}=3$ and
    $$B_1=0\Leftrightarrow \boldsymbol{b}=0^4.$$
    \item $\rho(-1)=0$ and $\rho(-2)=1$ imply $\boldsymbol{a}=3, 2$ and
    $$B_2=4\Leftrightarrow \boldsymbol{b}=2, 0^4.$$
     \item $\rho(-1)=1$ and $\rho(-2)=0$ imply $\boldsymbol{a}=3, 1$ and
    $$B_2=1\Leftrightarrow \boldsymbol{b}=1, 0^4.$$
    \item $\rho(-1)=\rho(-2)=1$ imply $\boldsymbol{a}=3, 2, 1$ and
    $$B_3=5\Leftrightarrow \boldsymbol{b}=2, 1, 0^4.$$
    \item $\rho(-1)=2$ and $\rho(-2)=0$ imply $\boldsymbol{a}=3, 1^2$ and
    $$B_3=2\Leftrightarrow \boldsymbol{b}=1^2, 0^4.$$
    \item $\rho(-1)=2$ and $\rho(-2)=1$ imply $\boldsymbol{a}=3, 1^2, 2$ and
    $$B_4=6\Leftrightarrow \boldsymbol{b}=2, 1^2, 0^4.$$
\end{itemize}
$\left(\mathbf{\X_6^{9}}, \rho(-4)=1\right)$
\begin{itemize}
    \item $\rho(-1)=0$ imply $\boldsymbol{a}=4$ and $B_1=7$ which has no solution.
    \item $\rho(-1)=1$ imply $\boldsymbol{a}=4,1$ and
    $$B_2=8\Leftrightarrow \textcolor{blue}{\boldsymbol{b}=2^2, 0^2}.$$
    \item $\rho(-1)=2$ imply $\boldsymbol{a}=4,1^2$ and
    $$B_3=9\Leftrightarrow \textcolor{blue}{\boldsymbol{b}=2^2,1,0^2} \mbox{ or } \boldsymbol{b}=3, 0^6.$$
    
\end{itemize}
%%%%%
%%%%%%
$\left(\mathbf{\X_7^{9}}, \rho(-2)=4\right)$
\begin{itemize}
    \item $\rho(-1)=0$ imply $\boldsymbol{a}=2^4$ and
    $$B_4=7\Leftrightarrow \textcolor{red}{\boldsymbol{b}=2, 1^3, 0^2}.$$
    
\end{itemize}
$\left(\mathbf{\X_8^{9}}, \rho(-3)=1\right)$
\begin{itemize}
    \item $\rho(-2)=1$ imply $\boldsymbol{a}=3, 2$ and
    $$B_2=4\Leftrightarrow \boldsymbol{b}=2, 0^4.$$
\item $\rho(-2)=2$ imply $\boldsymbol{a}=3, 2^2$ and
    $$B_2=8\Leftrightarrow \textcolor{blue}{\boldsymbol{b}=2^2, 0^4}.$$    
\end{itemize}
%%
%%
%%%%%
%%%%%%
$\left(\mathbf{\X_9^{9}}, \rho(-3)=2\right)$
\begin{itemize}
  \item $\rho(-2)=0$ imply $\boldsymbol{a}=3^2$ and
  $$B_2=9\Leftrightarrow \textcolor{red}{\boldsymbol{b}=3, 0^4} \mbox{ or } \boldsymbol{b}=2^2, 1 .$$
  \item $\rho(-2)=1$ imply $\boldsymbol{a}=3^2,2$ and
  $$B_3=13\Leftrightarrow \boldsymbol{b}=2^3,1 \mbox{ or } \textcolor{red}{\boldsymbol{b}=3,2, 0^4}.$$
\end{itemize}
%%%%%
%%%%%%
$\left(\mathbf{\X_{10}^{9}}, \rho(-4)=1\right)$
\begin{itemize}
    \item $\rho(-2)=0$ imply $B_1=7$ which has no solution.
   \item $\rho(-2)=1$ imply $\boldsymbol{a}=4, 2$ and
    $$B_2=11\Leftrightarrow \boldsymbol{b}=3, 1^2.$$
\end{itemize}
%%%
%%%%
\begin{Lemma}\label{nonstable22}
Let $b_1, b_2, b_3$ and $a_1, a_2$ be integers such that 
\begin{enumerate}
\item $a_2>b_1\geq b_2\geq b_3>a_1\geq0$.
\item $b_1+b_2+b_3\geq a_2$.
\end{enumerate}

The vector 
bundle $\EE$ on $\mathbb{P}^3$ that is cohomology of the minimal monad
\begin{equation*}
\op3(-a_2)\oplus g\cdot\op3(-a_1) \stackrel{\alpha}{\longrightarrow}
\displaystyle\bigoplus_{i=1}^3\op3(b_i)\oplus B \stackrel{\beta}{\longrightarrow}
\op3(a_2)\oplus g\cdot\op3(a_1).
\end{equation*}
is not stable, where $B$ is a ($2g+1$)-vector bundle.
\end{Lemma}

\begin{proof}
Let's consider $K=\ker\beta$ and $K''=\ker\psi$. If $\EE$ is stable and we write the morphism $\beta$  of the monad as follows
$$\beta=\left(
\begin{array}{cc}
    \psi & \eta \\
   0  & \phi
\end{array}
\right),$$
we get the commutative diagram  
$$\xymatrix{
 K''\ar@{^{(}->}[d]\ar[r]& \displaystyle\bigoplus_{i=1}^3\op3(b_i) \ar@{^{(}->}[d]\ar[r]^{\psi} & \op3(a_2)\ar@{^{(}->}[d] \\
 K\ar[r]\ar[d]& \displaystyle\bigoplus_{i=1}^3\op3(b_i)\oplus B \ar[r]^{\beta} \ar[d]& \op3(a_2)\oplus g\cdot\op3(a_1)\ar[d].\\
 K'\ar[r]&B\ar@{->>}[r]^{\phi}& g\cdot\op3(a_1).\\
}$$
Since $K''\simeq\ker\{K\longrightarrow K'\}$ with $K, K'$ locally free sheaves and $\phi$ is surjective, follows that 
$K''$ is a stable reflexive sheaf of rank two such that $c_1(K'')=b_1+b_2+b_3-a_2\geq0=c_1(\EE)$. From 
\cite[Proposition 10]{MF2022}, the morphism $K''\longrightarrow K\longrightarrow\EE$ is 
zero and so it can be extended to a morphism $K''\longrightarrow\op3(-a_2)\oplus g\cdot\op3(-a_1)$ which is a contradiction because
$$\begin{array}{c}
\Hom(K'',\op3(-a_2)\oplus g\cdot\op3(-a_1))\simeq {H^3(\mathbb{P}^3, K''(a_2-4))}^{*}\oplus g{H^3(\mathbb{P}^3, K''(a_1-4))}^{*}\\
\simeq H^0(K''^{*}(-a_2))+gH^0(K''^{*}(-a_1))=0,\\
%\simeq H^0(K''(-b_1-b_2-b_3))+gH^0(K''(a_2-a_1-b_1-b_2-b_3))=0,
\end{array}$$
since $K''$ is a stable reflexive sheaf with $K''^{*}\simeq K''(a_2-b_1-b_2-b_3)$.
\end{proof}
%%
%\begin{prop}
%Let's consider the following possibilities of positive minimal Horrocks monads: $(\i) \boldsymbol{a}=3, 1^3$ and 
%$\boldsymbol{b}=1^3, 0^4$; $(ii) \boldsymbol{a}=4, 1^2$ and $\boldsymbol{b}=2^2, 1, 0^2$. If one of these monad 
%possibilities exists and $\EE$ is its cohomology, then $\EE$ is not a stable bundle. 
%\end{prop}
%\begin{proof}
 %The non-stability of the vector bundle given by the 
%cohomology of one of the monads $(\iii)$ and $(\iv)$ follows from Lemma \ref{nonstable22} by 
%considering $a_2=3, a_1=1, g=3, b_1=b_2=b_3=1$ and $a_2=4, a_1=1, g=2, b_1=b_2=2, b_3=1$, respectively.
%\end{proof}
%%%%
\begin{prop}
Let us consider the possible positive minimal monads in blue as listed above:  $(\i) \boldsymbol{a}=3, 1^4$ and 
$\boldsymbol{b}=1^4, 0^4$; $(\ii) \boldsymbol{a}=4, 1$ and 
$\boldsymbol{b}=2^2, 0^2$; $(\iii) \boldsymbol{a}=3, 2^2$ and $\boldsymbol{b}=2^2, 0^4$; $(\iv) \boldsymbol{a}=3, 1^3$ and $\boldsymbol{b}=1^3, 0^4$; $(\v) \boldsymbol{a}=4, 1^2$ and $\boldsymbol{b}=2^2,1, 0^2$; If one of these monad 
possibilities exists and $\EE$ is its cohomology, then $\EE$ is not a stable bundle. 
\end{prop}
\begin{proof}
We will prove that the cohomology of the first possible of minimal Horrocks monad is an unstable bundle. Let us consider the existence of the minimal Horrocks monad 
$$\op3(-3)\oplus 4\cdot\op3(-1) \stackrel{\alpha}{\longrightarrow}
4\cdot\op3(1)\oplus 4\cdot\op3\oplus4\cdot\op3(-1) \stackrel{\beta}{\longrightarrow}
\op3(3)\oplus 4\cdot\op3(1),$$
 and $\EE$ its cohomology. We suppose, by contradiction, that $\EE$ is a stable 
 bundle. We have the commutative diagram
 $$\xymatrix{
 K'\ar@{^{(}->}[d]\ar[r]& 4\cdot\op3(1) \ar@{^{(}->}[d]\ar[r]^{\beta'} & \op3(3)\ar@{^{(}->}[d] \\
 K\ar[r]\ar[d]& 4\cdot\op3(1)\oplus 4\cdot\op3\oplus4\cdot\op3(-1)\ar[r] \ar[d]& 4\cdot\op3(1)\oplus\op3(3)\ar[d].\\
 K''\ar[r]&4\cdot\op3\oplus4\cdot\op3(-1)\ar@{->>}[r]^{\beta''}& 4\cdot\op3(1).\\
}$$
The morphism $\beta':4\op3(1)\rightarrow\op3(3)$ is surjective which implies $K'$ be a rank 3 reflexive 
sheaf with $\mu(K')=\frac{1}{3}$. Furthermore, the reflexive sheaf $K'$ is 
stable c.f. \cite[Remark 1.2.6]{Oko80} and 
$$H^0(\mathbb{P}^3, K'(-1))=H^0(\mathbb{P}^3, {(K')}^{\vee})=0.$$
On the other hand, we have the commutative diagram
$$\xymatrix{
 & K' \ar@{^{(}->}[d]\ar@{-->}[rd]^{\Psi}& \\
 \op3(-3)\oplus4\cdot\op3(-1)\ar[r]& K \ar[r]& \EE,\\
}$$
where we obtain the morphism $\Psi: K'\rightarrow K\rightarrow\EE$ which is zero c.f. 
\cite[Proposition 10]{MF2022} because $\mu(K')>\mu(\EE)=0$ and this $\Psi$ can be 
extended to a morphism $K'\rightarrow\op3(-3)\oplus4\cdot\op3(-1)$ that is a 
contradiction, since $K'$ is stable implies 
$$\Hom(K',\op3(-3)\oplus4\cdot\op3(-1))=0.$$

In the second case we consider the morphism $\beta':2\cdot\op3(2)\rightarrow\op3(4)$ while in the third case we take 
$\beta':2\cdot\op3(2)\rightarrow\op3(3)$ and so we repeat the above argument. Finally, in the case 
$(iv)$ we consider $\beta':3\cdot\op3(1)\rightarrow\op3(3)$ and in the case $(v)$ the 
morfism $\beta':2\cdot\op3(2)\oplus\op3(1)\rightarrow\op3(4)$. 
\end{proof}
%%%
\begin{prop}\label{nosection}
 If the three possible positive minimal Horrocks monads in red exist, then its cohomology is an unstable bundle.   
\end{prop}
\begin{proof}
    The possible positive minimal monads in red are: $\boldsymbol{a}=2^4, \boldsymbol{b}=2, 1^3, 0^2$; $\boldsymbol{a}=3^2, \boldsymbol{b}=3,0^4$; $\boldsymbol{a}=3^2,2, \boldsymbol{b}=3,2, 0^4$. We suppose the first possible minimal monad exists and its cohomology $\EE$ is a stable bundle. We rewrite this monad as
    $$4\cdot\op3(-2)\stackrel{\alpha}{\longrightarrow}
\op3(2)\oplus 3\cdot\op3(1)\oplus2\cdot\op3\oplus\op3(-2)\oplus3\cdot\op3(-1) \stackrel{\beta}{\longrightarrow}
 4\cdot\op3(2),$$
 and from the minimally of the monad the first column of $\beta$ is zero. If we take $K=\ker\beta$ then, for example, 
 $\eta={\left(\begin{array}{cccccccc}x^2& 0 &0& 0& 0 &0& 0 &0\end{array}\right)}^t\in H^0(K)=H^0(\EE)$ which is 
 contradiction. With a similar argument,t we show that the cohomology of the other two monads is unstable.
 \end{proof}
%%%%%
%%%%%
\section{Existence of minimal Horrocks monads}\label{section4}
The goal of this section is to list all the existing minimal monads whose cohomology is a stable rank 2 bundle on $\mathbb{P}^3$ with $c_1=0, c_2=9$. From Theorem \ref{teo4} we will conclude that there are no negative monads in this case and there are at most three non-negative minimal monads.
\subsection{Negative and Nonnegative minimal Horrocks monads}\label{sec4}

%\begin{Lemma}\label{negative}
%Let $\mathbf{M}: \mathcal{A}^*\longrightarrow\mathcal{B}\longrightarrow\mathcal{A}$ be a monad 
%as in \eqref{eq6}. If $\mathcal{A}$ has $r$ summands with degree 
%$\leq l$, then $\mathcal{B}$ must contain at least $r+3$ summands of degrees $\geq 1-l$.
%\end{Lemma}

With the notation of this paper, in \cite{coanda2024}, Iustin Coanda proved the following assertions.
\begin{Theorem}[Theorem 2.7]\label{teo4}
  \begin{enumerate}
      \item $\rho(i)\leq s(-i-1)-1$, for $-m\leq i\leq m-1$.
      \item If $\rho(i)=s(-i-1)-1$, for some $i$ with $-m\leq i\leq -2$, then $s(j)=1$ for $-i\leq j\leq m$ 
      and there is a plane $H_0\subset\mathbb{P}^3$ such that $H^0(E_{H_0}(-m))\neq0$.
      \item\label{3} $\rho(i)\leq\max(s(-i-1)-2,0)$, for $i\geq0$.
  \end{enumerate}  
\end{Theorem}
By applying the item \ref{3} of Theorem \ref{teo4} to each spectrum of the Table 
\ref{spectra} we can verify that there are no negative minimal Horrocks monad. Furthermore,
if we apply the item \ref{3} again with $i=0$ to each spectrum of the table \ref{spectra} we see that
the spectra $\mathcal{X}_7^9=\{-1^4,0,1^4\}$ and $\mathcal{X}_8^9=\{-2,-1^3,0,1^3,2\}$ 
provide possibilities of no-negative minimal monads. 
\begin{prop}
 Let $\mathcal{E}$ be a stable rank 2 bubdle on $\mathbb{P}^3$ with spectrum $\mathcal{X}_7^9=\{-1^4,0,1^4\}$. The only 
 no-negative minimal Horrocks monad whose cohomology has spectrum $\mathcal{X}_7^9$ is
 $$\op3(-2)^{\oplus4}\oplus\op3^{\oplus2}\stackrel{\alpha}{\longrightarrow}
\op3(1)^{\oplus7}\oplus\op3(-1)^{\oplus7} \stackrel{\beta}{\longrightarrow}
 \op3(2)^{\oplus4}\oplus\op3^{\oplus2}.$$
\end{prop}
\begin{proof}
For $\mathcal{X}_7^9$ we have $\rho(-2)=4$ 
and $\rho(0)\leq2$. From Equation \eqref{c_2}, if $\rho(0)=1$ then
$$\displaystyle\sum_{i=1}^{6}b_i^2=7\Rightarrow b_1=2, b_2=b_3=b_4=1, b_5=b_6=0.$$
If this minimal monad exists, then by repeating the argument of the Proposition \ref{nosection} we conclude that 
its cohomology is not a stable bundle. If $\rho(0)=2$ then
$$\displaystyle\sum_{i=1}^{7}b_i^2=7\Rightarrow b_1=2, b_2=b_3=b_4=1, b_5=b_6=b_7=0 \mbox{ or } b_i=1, i=1,\cdots,7.$$
If the first possible of minimal monad exists, then its cohomology $\mathcal{E}$ is not a stable bundle. For the values
$\rho(-2)=4, \rho(0)=2$ and $b_i=1, i=1,\cdots,7$ we consider a smooth curve $X$ of bidegree
$(1,4)$ on a nonsingular quadric surface $Q$ and from the exact sequence
$$0\rightarrow\mathcal{O}_Q(-1,-4)\rightarrow\mathcal{O}_Q\rightarrow\mathcal{O}_X\rightarrow0,$$
we have c.f. \cite[Proposition 6.1]{HR91} $h^0(\mathcal{I}_X(1))=2$. From Hartshorne-Serre correspondence in Subsection
\ref{serre}, the bundle $\mathcal{E}$ corresponding  to $X$ is stable with $c_1=0, c_2=9$ and 
$\mathcal{X}(\mathcal{E})=\mathcal{X}_7^9$.
\end{proof}
From item \ref{3} of Theorem \ref{teo4}, the spectrum $\mathcal{X}_8^9=\{-2,-1^3,0,1^3,2\}$ gives
$\rho(-3)=1, \rho(-2)\in\{1,2\}$ and $\rho(0)\leq1$.
\begin{itemize}
    \item $\rho(-3)=1, \rho(-2)=1$ and $\rho(0)=1$. In this case, 
    $$\displaystyle\sum_{i=1}^{4}b_i^2=4\Rightarrow b_1=2, b_2=b_3=b_4=0 \mbox{ or } b_i=1, i=1,2,3,4.$$
    \item $\rho(-3)=1, \rho(-2)=2$ and $\rho(0)=1$. In this case, 
    $$\displaystyle\sum_{i=1}^{5}b_i^2=8\Rightarrow b_1=b_2=2, b_3=b_4=b_5=0 \mbox{ or } b_1=2, b_2=b_3=b_4=b_5=1 .$$
\end{itemize}
The 2 possibilities of minimal monads $b_1=2, b_2=b_3=b_4=0$ and $b_1=b_2=2, b_3=b_4=b_5=0$ are
eliminated from \cite[Lemma 3.3]{HR91}. 

\subsection{Positive minimal monads}
%%%%%
The goal of this subsection is to list all positive minimal Horrocks monads whose cohomology 
is a stable rank two vector bundle on $\mathbb{P}^3$ with $c_2=9, c_1=0$. We will enunciate in this section two methods 
of building minimal Horrocks monads whose cohomology is a stable bunble with even determinant 
and $c_2=9$, see also \cite{HR91}. Let us
denote by (a) the first form of construction of monads. To recall this type of construction 
of monads its necessary to consider a locally complete intersection 
curve $X$ in $\mathbb{P}^3$ of degree $d$ such that each connected component has no global section 
and the rank 2 sheaf $\mathcal{N}_X\otimes\omega_X(2)$ has a nowhere vanish section 
$s$ (this is true by \cite[Proposition 2.8]{HR91}). Then we consider the curve $Y$ with a multiplicity 
2 structure on $X$ given by the kernel the the morphism $\mathcal{I}_X\longrightarrow\omega_X(2)$ such that we have the exact sequence
\begin{equation}\label{eq2.3}
0\longrightarrow\mathcal{I}_Y\longrightarrow\mathcal{I}_X\longrightarrow\omega_X(2)\longrightarrow0.
\end{equation}
By Ferrand's Theorem we get $\omega_Y=\mathcal{O}_Y(-2)$ and from Hartshorne-Serre 
corresponde, see \cite[Theorem 4.1]{H80}, $Y$ is a zero scheme of a section of $\EE(1)$ where $\EE$ is a 
stable bundle with $c_1=0, c_2=2d-1$ and there is the exact sequence
\begin{equation}\label{eq2.4}
0\longrightarrow\mathcal{O}\longrightarrow\EE(1)\longrightarrow\mathcal{I}_Y(2)\longrightarrow0.
\end{equation}
%%%%%
The second form (b) of construction of monads consists of the application of the following Lemma 
which also can be found in \cite[Lemma 4.8]{HR91}.
\begin{Lemma}\label{lema48}
Let $(E,\sigma)$ be a pair consisting of a stable rank 2 vector bundle $E$ with $c_1(E)=-1$ and $c_2(E)=n$ 
and a section $\sigma\in H^0(E(r))$ with $r>0$ such that $X:=(\sigma)_0$ is a curve. If $C$ is a complete 
intersection curve of type $(a,b)$ disjoint from $X$ satisfying $a+b=2r-1$, then there is a pair 
$(E',\sigma')$ consisting of a stable rank 2 vector bundle $E$ with $c_1(E)=0$ and $c_2(E)=n+ab$ 
and a section $\sigma'\in H^0(E'(r))$ such that $(\sigma')_0=X\sqcup C$. Moreover, if $E$ 
is the cohomology of a minimal monad of the form
$$ \mathbf{M}:\ \ \  \ \mathcal{C} \longrightarrow \mathcal{B} \longrightarrow \mathcal{A}, $$
then $E'$ is the cohomology of a minimal monad of the form
$$ \mathbf{M}':\ \ \op3(-r)\oplus\mathcal{C} \longrightarrow \op3(r-a)\oplus\op3(r-b)\oplus\mathcal{B} \longrightarrow \op3(r)\oplus\mathcal{A}. $$
\end{Lemma}
%%%%
\begin{Remark}
The positive minimal Horrocks monads obtained from type (b) are easily identified once we choose the curve, while
monads found of type (a) are not immediately identified when choosing the curve.
\end{Remark}
\begin{prop}\label{existence1}
The possible positive minimal Horrocks monad with $\boldsymbol{b}=3,1^2$ and $\boldsymbol{a}=4, 2$ exists.
\end{prop}
\begin{proof}
Let's apply the construction (a) and for this, we take the curve $X:=P_4\cup P_1$ where $P_4$ and 
$P_2$ are plane curves of degree 4 and 2, respectively, joined at a point $x_0$. Applying construction (a),
we obtain a stable vector bundle $\EE$ such that $c_1=0, c_2=9$. Now we computer the spectrum of $\EE$. To computer 
the dualizing sheaf $\omega_X$ we recall the exact sequence
\begin{equation*}
0\longrightarrow\omega_{P_4}\oplus\omega_{P_1}\longrightarrow\omega_X\longrightarrow\omega_{\{x_0\}}\longrightarrow0.    
\end{equation*}
We know that $\omega_{P_4}=\mathcal{O}_{P_4}(1)$, $\omega_{P_1}=\mathcal{O}_{P_1}(-2)$ 
and $h^0(\omega_{\{x_0\}}(m))=1, m\in\mathbb{Z}$ and from exact sequences \eqref{eq2.3} and \eqref{eq2.4} follows that
$$h^1(\EE(-1))=12, h^1(\EE(-2))=7, h^1(\EE(-3))=3, h^1(\EE(-4))=1$$
and $h^1(\EE(-l))=0, l\geq5$. Therefore, $\X=\{\ldots, 0, 1^2, 2,3\}$ and by Table \ref{terms} the only 
possible of positive monad has $\boldsymbol{b}=3,1^2$ and $\boldsymbol{a}=4, 2$.
\end{proof}
%%\

\begin{prop}
The positive minimal Horrocks monads $\mathbf{M}$ such that $\boldsymbol{b}=2^2,1^2$ 
 and $\boldsymbol{a}=3^2, 2$ and $\mathbf{M'}$ with $\boldsymbol{b}=2^2,1$ 
 and $\boldsymbol{a}=3^2$ exist. Furthermore, its cohomologies  are stable bundles.
\end{prop}
\begin{proof}
If we take $X:=P_2\cup P_3$ a union of planes curves of degree $2$ and $3$, respectively, joined at two 
point $\{x_0, x_1\}$ then the construction (a) provides a stable bundle $\mathcal{F}$ with 
$c_1=0, c_2=9$ and repeating the argument of the proof of the 
Proposition \ref{existence1} we obtain the spectrum $\X_9^9$. Furthermore,
\begin{equation*}
\begin{array}{l}
 H^1(\mathcal{F}(-3))\simeq H^0(\omega_X)\simeq H^0(\mathcal{O}_{P_3})\oplus H^0(k(x_0))\\
 H^1(\mathcal{F}(-2))\simeq H^0(\mathcal{O}_{P_2})\oplus H^0(\mathcal{O}_{P_3}(1))\oplus H^0(\omega_S(1)).\\
 \end{array}
 \end{equation*}
 Therefore, $\rho(-2)=1$. This is sufficient to conclude that $\mathbf{M}$ is the positive monad whose 
 cohomology is the stable bundle $\mathcal{F}$. Now 
 let $X:=C_{3,2}$ be a curve of type $(3,2)$ on a smooth quadric $Q$ where we have the exact sequence
 \begin{equation*}
0\longrightarrow\mathcal{O}_Q(-2,-2)\longrightarrow\mathcal{O}_Q(1,0)\longrightarrow\omega_X\longrightarrow0.
 \end{equation*}
 By the construction (a) there is a stable bundle on $\mathbb{P}^3$ such that
 $$ H^1(\mathcal{\EE}(-3))\simeq H^0(\mathcal{O}_Q(1,0)) \mbox{ and } H^1(\mathcal{\EE}(-2))\simeq H^0(\mathcal{O}_Q(1,0)(1)).$$
 Therefore, we get $\rho(-2)=0$ and from Table \ref{terms} we obtain that $\mathbf{M'}$ is the minimal positive 
 Horrocks monad whose cohomology is the stable bundle $\EE$.
 \end{proof}
%%%%%%

In Table \ref{table:c2=10} it will be tabulating all positive and negative minimal Horrocks monads, with the 
possible exceptions of $\boldsymbol{a}=(3,2,0), \boldsymbol{b}=(1^4)$ and 
$\boldsymbol{a}=(3,2^2, 0), \boldsymbol{b}=(2, 1^4)$,   
whose cohomology is a rank $2$ vector bundle on $\mathbb{P}^3$ with even determinant and $c_2=9$, . Let's 
establish the following notation on the Table: \\
\begin{enumerate}
    \item[(1)] The fourth column indicates the degree in which the global section exists for constructing the vector 
    bundle $\EE$, giving a curve $Y$ as a zero scheme.
    \item[(2)] The last column describe as we construct the curve $X$ by applying the Lemma \ref{lema48}.
    \item [(3)] $P_n$ denotes a plane curve of degree $n$ while $C_{a,b}$ denotes a curve of bidegree $(a,b)$ on a smooth quadric.
\end{enumerate}

%%%%%%%%%%%%

%%%%%%%
\newpage
%%%%%%%%
	\begin{table}[hht]
		\begin{center}
	\begin{tabular}{ | l | c | c | c |c|} 
\hline
Spectrum & $\boldsymbol{b}$ &  $\boldsymbol{a}$ &$r$ &Construction\\ 
\hline
$\X^{9}_1$ & $0^{20}$ & $1^9$ &$1$ & Instanton\\
\hline
\multirow{2}{4em}{$\X^{9}_2$}&$0^{14}$ &$2,1^5$& $\leq2$& (b): $8(2, \i), P_1$ or (b): $5(1), C_{2,2}$\\
&$1, 0^{14}$& $2,1^6$& $1$ &(b): $8(2, \ii), P_1$\\
\hline
\multirow{4}{4em}{$\X^{9}_3$}& $0^8$& $2^2, 1$&$1$ &(b): $8(3, \i), P_1$\\
& $1, 0^8$& $2^2, 1^2$ &$1$&(b): $8(3, \ii), P_1$\\
& $1^2, 0^8$& $2^2, 1^3$ &$1$&(b): $8(3, \iii), P_1$\\
& $1^3, 0^8$& $2^2, 1^4$&$1$&(b): $8(3, \iv), P_1$\\
\hline
\multirow{4}{4em}{$\X^{9}_4$}& $1, 0^4$& $3, 1$&$1$&(b): $8(6, \i), P_1$\\
& $1^2, 0^4$&$3,1^2$ &1&(b): $8(6, \ii), P_1$\\
%& $1^3, 0^4$& $3,1^3$ &-&-\\
& $2,0^{10}$&$3,1^4$ &$1$&(b): $8(5), P_1$\\
\hline
\multirow{6}{4em}{$\X^{9}_5$}& $0^4$& $3$&$3$ &Ein\\
& $2, 0^4$&$3,2$ &$2$&(b): $5(4), C_{2,2}$\\
& $1, 0^4$& $3,1$ &$1$&(b): $8(6, \i), P_1$\\
& $2,1, 0^{4}$&$3,2,1$ &$1$&(b): $8(6, \iii), P_1$\\
& $1^2, 0^4$& $3,1^2$ &$1$&(b): $8(6, \ii), P_1$\\
& $2,1^2, 0^{4}$&$3,2,1^2$ &$1$&(b): $8(6, \iv), P_1$\\
\hline
$\X^{9}_6$& $3, 0^6$&$4, 1^2$ & $1$&(b): $8(7), P_1$\\
\hline
$\X^{9}_7$&$1^7$& $2^4,0^2$&-&$C_{1,4}$\\
\hline
\multirow{3}{4em}{$\X^{9}_8$}& $2,0^4$& $3,2$&$2$&(b): $5(2), P_1$\\
& $1^4$& $3,2,0$& ?&?\\ 
& $2,1^4$& $3,2^2,0$& ?&?\\
 \hline
 \multirow{2}{4em}{$\X^{9}_{9}$}& $2^2,1$& $3^2$&$1$& (a): $C_{3,2}$\\
 & $2^3,1$&$3^2, 2$&$1$ &(a): $P_2\cup P_3$ joined at two points\\
 \hline
$\X^{9}_{10}$& $3,1^2$& $4, 2$& $1$ & (a): $P_4\cup P_1$ joined at a point\\
\hline
$\X^9_{11}$ & $4, 0^2$ & $5$& $1$& (a): $P_5$\\
\hline

	\end{tabular}
\caption{Positive and nonegative minimal Horrocks monads whose cohomology is a stable bundle with even determinant and $c_2=9$.}
\label{table:c2=10}
\end{center}
\end{table}
%%%\
\section{Irreducible components of $\mathcal{B}(n)$ for $n\leq9$}\label{sec5}
The goal of this Section is to produce a new infinite series 
of components of $\mathcal{B}(n)$. A complete description of $\mathcal{B}(n)$ is known 
for $n\leq5$. More precisely, the moduli 
space $\mathcal{B}(1)$ was studied by Bath in \cite{B77} which is irreducible 
of dimension $5$ (see also \cite{H78}) while the moduli space $\mathcal{B}(2)$ was 
described by Hartshorne in \cite{H78} where he shows that this space is irreducible of 
dimension $13$; The moduli space $\mathcal{B}(3)$ was described in \cite{ES} which is 
nonsingular and has two irreducible components, both of dimension $21$. We remember that the moduli 
space $\mathcal{B}(4)$ has two irreducible components of dimension $29$, see \cite{Bar81} 
and \cite{C83}. Recently, 
in \cite{AJTT} is provided a characterization of $\mathcal{B}(5)$ which has three irreducible 
components, two components of dimension 37 and one of dimension 40. Two families of irreducible 
components of $\mathcal{B}(n), n\geq1$ have been studied, the so-called \textit{instanton component} whose 
generic point corresponds to an instanton bundle and the \textit{Ein component} where a generic point 
is a generalized null correlation bundle c.f. \cite{Ein88}. It is important to highlight that all irreducible components 
of $\mathcal{B}(n)$ for $n\leq4$ are of either of these types. 

In \cite{TV2019} is constructed a new infinite series $\Sigma_0$ whose elements are 
irreducible components $\mathcal{M}_n$ of the moduli space $\mathcal{B}(n)$, distinct from the series of 
instanton components and from the series of Ein components, and a generic bundle of 
these components is given as cohomology sheaf of a monad whose middle 
term is a rank 4 sympletic instanton bundle. The family of irreducible components $\mathcal{M}_n$ of $\mathcal{B}(n)$ of 
\cite[Theorem 1]{TV2019} generalizes the family of irreducible components 
$\mathcal{G}(k,1)$ in \cite{AJTT} with $n=1+k^2, k\in\{2\}\cup(4,\infty)$.
%%%%

%%%
\begin{Lemma}
Let $a$ be an integer such that $a\geq3$. The minimal Horrocks monad $\mathbf{M}$ given by
\begin{equation}\label{monad1}
2\cdot\op3(-a)\stackrel{\alpha}\rightarrow2\cdot\op3(a-1)\oplus2\cdot\op3(1-a)\oplus(\op3(1)\oplus\op3(-1))\stackrel{\beta}\rightarrow2\cdot\op3(a)
\end{equation}
exists and its cohomology $\EE$ is a stable bundle with $c_1(\EE)=0, c_2(\EE)=4a-3$ and spectrum
\begin{equation}\label{spectrum1}
 \mathcal{X}=\{{(1-a)}^2,\cdots,{-1}^2,0,1^2,\cdots,{(a-2)}^2, {(a-1)}^2\}.   
\end{equation}
\end{Lemma}
\begin{proof}
We prove the existence of the monad \eqref{monad1} by an explicit computation. In fact, with the help of the Macaulay2 we 
found the morphisms $\alpha$ and $\beta$ as
$$\beta=\left(
\begin{array}{cccccc}
x& 0&w^{2a-1}& z^{2a-1}& y^{a-1}& z^{a+1}\\
y& x& z^{2a-1}& w^{2a-1}& x^{a-1}& 0\\
\end{array} \right)
~\mbox{ and }~\alpha=\left(
\begin{array}{cc}
w^{2a-1}&0\\
z^{2a-1}&w^{2a-1}+x^{a-2}z^{a+1}\\
-x&0\\
-y&-x\\
0&-z^{a+1}\\
yz^{a-2}&xz^{a-2}+y^{a-1}\\
\end{array}
\right),$$
such that $\alpha$ is injective, $\beta$ is surjective and $\beta\circ\alpha=0$. This is sufficient 
to show that the minimal Horrocks monad in \eqref{monad1} exists and therefore its cohomology is 
a vector bundle $\EE$ such that $c_1(\EE)=0, c_2(\EE)=4a-3$ and furthermore, the morphism $\beta$ does not 
admit syzygies of degree $\geq0$ and so the vector bundle $\EE$ is stable. On the other hand, from 
display of the monad in \eqref{monad1} we have for $l\geq2$
$$n_l=2h^0(\op3(a-l))-2h^0(\op3(a-l-1))=(a+2-l)(a+1-l) \mbox{ and } n_1=2a-1.$$
By applying the identity in \eqref{S4} we prove that the spectrum of $\EE$ is as in \eqref{spectrum1}.

\end{proof}
Now we make some considerations to computer the dimension of the family of minimal monads of the form 
\eqref{monad1} which can be found with more details in \cite{BH78} and \cite{MF2021}. Define 
$\mathcal{P}(\boldsymbol{a})$ as the family of minimal 
Horrocks monads of the form \eqref{monad1} with $\mathcal{A}:=2\cdot\op3(a)$ and 
$\mathcal{B}$ the middle term of a monad of the form 
\eqref{monad1}, which are homotopy free. Let's denote $\mathcal{V}(\boldsymbol{a})$ to be the set of isomorphism 
classes of stable vector bundles given as cohomology of minimal Horrocks monads of 
$\mathcal{P}(\boldsymbol{a})$. From \cite[Section 8]{MF2021}, the dimension of $\mathcal{V}(\boldsymbol{a})$ is given by formula 
\begin{equation} \label{dim formula}
\dim \mathcal{V}(\boldsymbol{a})=\dim H-\dim W-\dim\gl(\mathcal{A})-\dim G,
\end{equation}
where for this family we have
$\begin{array}{l}
\dim H=16+2{a+2\choose3}+2{a+4\choose3}+4{2a+2\choose3};\\
\dim\gl(\mathcal{A})=4;\\
\dim W=h^0(\op3(2a))={2a+3\choose3};\\
\dim G=15+2{a+1\choose3}+2{a+3\choose3}+3{2a+1\choose3}.\\
\end{array}$

By substituting in the Equation \eqref{dim formula} we obtain
$$\dim \mathcal{V}(\boldsymbol{a})=a(a+1)+(a+3)(a+2)+(2a+1)(2a-1)-3=6a^2+6a+2$$
Thus we can observe that the dimension of $\mathcal{V}(\boldsymbol{a})$ is larger than the 
expected dimension $8c_2-3=32a-27$ of the moduli space $\mathcal{B}(4a-3)$ because 
$$\dim \mathcal{V}(\boldsymbol{a})>32a-27\Leftrightarrow6a^2-26a+29>0,$$
which is always true since $a\geq3$.
%%%%%%%%
\begin{Theorem}
Let $\mathcal{B}(9)$ the moduli space of the stable rank 2 vector bundles on $\mathbb{P}^3$ with 
$c_1=0$ $c_2=9$. The moduli space $\mathcal{B}(9)$ has at least 4 components: a Hartshorne component $M_1$ 
of dimension 69, two Ein components $M_2, M_3$ of dimension $69$ and $96$, respectively, and a new component $M_4$ 
whose dimension is $\geq74$.
\end{Theorem}
\begin{proof}
The Hartshorne component has an expected dimension equal to $69$. A generic bundle of the Ein component $M_2$ is 
given as cohomology of a minimal Horrocks monad 
$$0\rightarrow\op3(-3)\rightarrow\op3^{\oplus4}\rightarrow\op3(3)\rightarrow0$$
while a generic bundle of the Ein component $M_3$ is given as cohomology of a minimal monad 
\begin{equation}\label{ein3}
 0\rightarrow\op3(-5)\rightarrow\op3(4)\oplus\op3^{\oplus2}\oplus\op3(-4)\rightarrow\op3(5)\rightarrow0.   
\end{equation}
The dimension of the Ein components $M_2$ and $M_3$ can be computed by formula \cite[2.2.B, p.g. 18]{Ein88} or by
Equation \eqref{dim formula} where we obtain $\dim M_2=69$ and $\dim M_3=96$. 

We take $a=3$ in $\eqref{monad1}$ and from equation \eqref{dim formula} we have 
$\dim \mathcal{V}(3^2;2,1)=74$. In this way, the family $\mathcal{V}(3^2;2,1)$ 
cannot be contained in either $M_1$ or $M_2$. On the other hand, if $\mathcal{F}$ is a generic point of 
$M_3$ then $h^1(\mathcal{F}(-3))=6$ and from inferior semicontinuity we have 
$h^1(\mathcal{G}(-3))\geq6, \forall \mathcal{G}\in M_3$. For $\mathcal{E}\in\mathcal{V}(3^2;2,1)$ we know 
that $h^1(\mathcal{E}(-3))=2$ and thus $\mathcal{V}(3^2;2,1)$ cannot be contained in 
$M_3$. Therefore, $\mathcal{V}(3^2;2,1)$ is contained in a new component $M_4$ of dimension at least $74$. 
\end{proof}
\begin{Remark}
 To prove that $\mathcal{V}(3^2;2,1)$ is not contained in $M_4$ we could observe that $M_4$ 
 has the form $N'(0,4,5)$ in \cite[Theorem 3.3]{Ein88} and therefore $\mathcal{V}(3^2;2,1)$ does not intersect other 
 irreducible component of $\mathcal{B}(9)$.  
\end{Remark}

%%%%

%%%%%%%%%%%%%%%%%%%%%%%%%%%%%%%%%%%%%%%%%%%%%%%%%%%%%%
%%%%%%%%%%%%%%%%%%%%%%%%%%%%%%%%%%%%%%%%%%%%%%%%%%%%%%
\newpage
\bibliographystyle{amsalpha}

\end{document}